\documentclass[11pt, reqno]{amsart}

\usepackage{amsmath,amssymb,amsthm, epsfig}
\usepackage{hyperref}
\usepackage{amsmath}
\usepackage{amssymb}
\usepackage{color}
\usepackage{ulem}
\usepackage{epsfig}
\usepackage[mathscr]{eucal}

\usepackage[utf8]{inputenc}

\usepackage{stackrel}

\usepackage{pstricks-add}
\usepackage{auto-pst-pdf}
\usepackage{pst-pdf}

\usepackage{epsfig}

\usepackage{pst-grad} 

\usepackage{pst-plot} 

\newtheorem{theorem}{Theorem}
\newtheorem{definition}{Definition}

\newtheorem{lemma}{Lemma}
\newtheorem{proposition}{Proposition}
\newtheorem{corollary}{Corollary}
\newtheorem{remark}{Remark}

\date{}
\numberwithin{equation}{section}
\numberwithin{theorem}{section}
\numberwithin{lemma}{section}
\numberwithin{corollary}{section}
\numberwithin{remark}{section}
\numberwithin{proposition}{section}
\numberwithin{definition}{section}

\def \Div {\mathrm{div}}

\def \R {\mathbb{R}}

\def \loc {\mathrm{loc}}

\begin{document}

\title[Sharp regularity estimates]{Sharp regularity estimates for quasilinear evolution equations}

\author[Amaral]{Marcelo D. Amaral}
\address{Department of Mathematics, Universidade da Integra\c{c}\~{a}o Internacional da Lusofonia Afro-Brasileira - UNILAB,  62785-000 Acarape, CE, Brazil}
\email{marceloamaral@unilab.edu.br}

\author[da Silva]{Jo\~{a}o V\'{i}tor da Silva}
\address{FCEyN, Department of Mathematics, University of Buenos Aires, Ciudad Universitaria-Pabell\'{o}n I-(C1428EGA) - Buenos Aires, Argentina}
\email{jdasilva@dm.uba.ar}

\author[Ricarte]{Gleydson C. Ricarte}
\address{Department of Mathematics, Federal University of Cear\'{a}, 60455-
760 Fortaleza, CE, Brazil}
\email{ricarte@mat.ufc.br}

\author[Teymurazyan]{Rafayel Teymurazyan}
\address{CMUC, Department of Mathematics, University of Coimbra, 3001-501 Coimbra, Portugal}
\email{rafayel@utexas.edu}

\begin{abstract}
We establish sharp geometric $C^{1+\alpha}$ regularity estimates for bounded weak solutions of evolution equations of $p$-Laplacian type. Our approach is based on geometric tangential methods, and makes use of a systematic oscillation mechanism combined with an adjusted intrinsic scaling argument.

\bigskip


\bigskip

\noindent \textbf{AMS Subject Classifications MSC 2010:} 35B65, 35K55, 35K65, 35J60, 35J70.
\tableofcontents
\end{abstract}

\maketitle

\section{Introduction}\label{s1}

In this work we obtain sharp geometric regularity estimates for bounded weak solutions of quasilinear parabolic equations (possibly singular and degenerate) of $p$-Laplacian type, whose prototype is given by
\begin{equation}\label{1.1}
u_t-\Div(|\nabla u|^{p-2}\nabla u) = f.
\end{equation}
In order to assure the existence of solutions in suitable Sobolev spaces (see \cite{AM07, AMS04, C91, DF85, KL00} for more details), $p$ is chosen such that
$$
\max\left\{1,\frac{2n}{n+2}\right\}<p<\infty.
$$

In our studies the source term $f$ is assumed to be in the \textit{anisotropic Lebesgue space} $L^{q,r}(\Omega_T)$ (where $\Omega\subset\R^n$ is an open, bounded set, $T>0$ and $\Omega_T:=\Omega\times(0,T)$), which is a Banach space endowed with the following norm:
$$
\|f\|_{L^{q, r}(\Omega_T)}:=\left(\int_{0}^{T} \left(\int_{\Omega}|f(x, t)|^qdx\right)^{\frac{r}{q}}dt\right)^{\frac{1}{r}}.
$$

Throughout the paper we will assume the following compatibility conditions:
\begin{equation}\label{cc}\tag{C}
\left\{
\begin{array}{ccc}
\frac{1}{r}+\frac{n}{pq}<1;\\
\max\left\{0;\left(1-\frac{1}{r}\right)(2-p)\right\} \leq \frac{2}{r}+\frac{n}{q}<1,
\end{array}
\right.
\end{equation}
where $q>n$ and $r>2$. The first inequality provides the minimal integrability condition, which guarantees the existence of bounded weak solutions of \eqref{1.1} (see, \cite[Ch.2, \S 1]{D93}). The second compatibility condition defines the fashion in which the gradient of weak solution has a universal H\"{o}lder modulus of continuity.

\begin{figure}[h]
\begin{center}
\psscalebox{0.6 0.6} 
{
\begin{pspicture}(0,-4.6)(13.2,4.6)
\psframe[linecolor=black, linewidth=0.04, dimen=outer](13.2,4.6)(0.0,-4.6)
\psline[linecolor=black, linewidth=0.04](1.6,3.4)(1.6,-3.0)
\psline[linecolor=black, linewidth=0.04](1.2,-2.6)(10.4,-2.6)
\psline[linecolor=black, linewidth=0.04, linestyle=dashed, dash=0.17638889cm 0.10583334cm](2.8,-2.6)(2.8,3.4)
\psline[linecolor=black, linewidth=0.04, linestyle=dashed, dash=0.17638889cm 0.10583334cm](1.6,-1.8)(9.2,-1.8)
\psline[linecolor=black, linewidth=0.04, linestyle=dashed, dash=0.17638889cm 0.10583334cm](1.6,-1.0)(9.2,-1.0)
\psbezier[linecolor=black, linewidth=0.04](3.2,3.4)(3.2,2.6)(3.05,-0.6)(4.325,-1.0)(5.6,-1.4)(8.45,-1.4)(9.2,-1.4)
\psbezier[linecolor=black, linewidth=0.04](4.8,3.4)(4.8,2.6)(5.309091,0.2)(6.523077,-0.2)(7.737063,-0.6)(9.538462,-0.6)(10.4,-0.6)
\psline[linecolor=black, linewidth=0.04, linestyle=dashed, dash=0.17638889cm 0.10583334cm](4.4,3.4)(4.4,-2.6)
\rput[bl](10.4,-3.0){$q$}
\rput[bl](1.2,3.0){$r$}
\rput[bl](4.2,-3.2){$\frac{n}{2-p}$}
\rput[bl](2.7,-3.1){$n$}
\rput[bl](10.6,-0.9){$\frac{\frac{n}{q}+\frac{2}{r}}{\frac{r-1}{r}}=2-p$}
\rput[bl](1.2,-1.9){$2$}
\rput[bl](0.9,-1.2){$\frac{4-p}{2-p}$}
\rput[bl](4.0,-0.87){$(2-p)\left(1-\frac{1}{r}\right)<\frac{n}{q}+\frac{2}{r}< 1$}
\psdots[linecolor=black, dotstyle=triangle*, dotsize=0.2](1.6,3.4)
\rput{31.265638}(0.19308585,-5.7662797){\psdots[linecolor=black, dotstyle=triangle*, dotsize=0.22](10.4,-2.6)}
\psline[linecolor=black, linewidth=0.04](9.2,-1.4)(10.4,-1.4)
\psline[linecolor=black, linewidth=0.04, linestyle=dashed, dash=0.17638889cm 0.10583334cm](9.2,-1.0)(10.4,-1.0)
\psline[linecolor=black, linewidth=0.04, linestyle=dashed, dash=0.17638889cm 0.10583334cm](9.2,-1.8)(10.4,-1.8)
\rput[bl](10.7,-1.6){$\frac{n}{q}+\frac{2}{r}=1$}
\end{pspicture}
}
\end{center}
\caption{ Regions of definition of $\frac{n}{q}+\frac{2}{r} \in \left[(2-p)\left(1-\frac{1}{r}\right), 1\right)$ for $\max\left\{1, \frac{2n}{n+2}\right\}<p<2$.}
\end{figure}
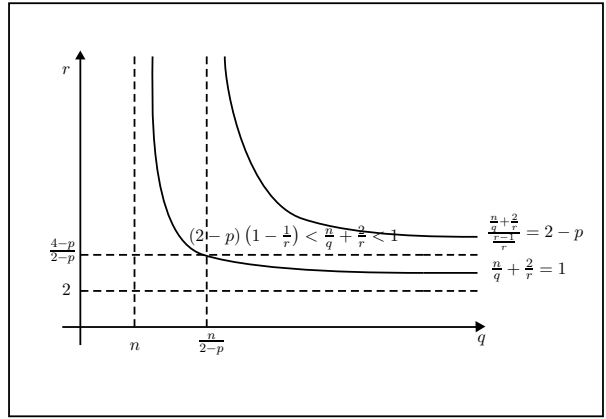

The regularity of weak solutions of quasilinear evolution problems received wide attention in the last decades (see, for example, \cite{AM07, AMS04, BC04, C91, DOS18, daSS18, D93, DF84, DF85, DF285, DiBUV02, K08, KM11, KM13, LSU68, L96, TU14, U08, W86}) due to its connection to a number of problems arising in biology, chemistry, mathematical physics, geometric and free boundary problems, etc (cf. \cite{BDM13} and \cite{DiBUV02} for complete essays on regularity of evolution equations with degenerate diffusion). Although weak solutions of \eqref{1.1} under the compatibility assumptions \eqref{cc} are known to be locally of the class $C^{1+\alpha}$ (in the parabolic sense) for some $\alpha \in (0, 1)$, the sharp exponent is known only for some specific cases (see \cite{ATU17, ATU18, IM89, K08, LL17}). This type of quantitative information plays an essential role in the study of blow-up analysis, related geometric and free boundary problems and for proving Liouville type results (see \cite{AT15, ART15, CLRT14, DOS18, daSS18, TU142} for some enlightening examples).

We recall (see \cite{D93,DiBUV02,LSU68,L96}) that when $p>2$ and
$$
  \frac{1}{r} + \frac{n}{pq}<1< \frac{2}{r}+ \frac{n}{q},
$$
weak solutions of \eqref{1.1} are of class $C^\alpha$ for some $\alpha\in(0,1)$. Using compactness and geometric tangential methods (see \cite{AT15, ART15, C89, CLRT14, DD18, DT17, T13, T14}) and intrinsic scaling techniques (see \cite{DOS18, D93, dBGV12, U08}), the sharp value of $\alpha$ was revealed in \cite[Theorem 3.4]{TU14}. The latter, however, leaves open issues in the following scenarios:
$$
  \frac{1}{r}+\frac{n}{pq}<1\,\,\textrm{ and }\,\,\frac{2}{r}+\frac{n}{q}=1
$$
and
$$
   \frac{1}{r}+\frac{n}{pq}<1 \,\,\textrm{ and }\,\, \frac{2}{r}+\frac{n}{q}<1.
$$
In this work we will solve it in the second scenario. More precisely, our main result reveals that bounded week solutions of \eqref{1.1} are locally of the class $C^{1+\alpha}$ (in the parabolic sense) in the critical zone (i.e. where gradient is small enough, see Section \ref{s3}), with
\begin{equation}\label{1.2}
\alpha:=\displaystyle\min\left\{\frac{1-\left(\frac{n}{q}+\frac{2}{r}\right)}{p\left[1-\left(\frac{n}{pq}+\frac{1}{r}\right)\right]-\left[1-\left(\frac{n}{q}+\frac{2}{r}\right)\right]}, \, \alpha_{\mathrm{H}}^{-} \right\},
\end{equation}
where $\alpha_{\mathrm{H}}\in(0,1]$ is the optimal regularity exponent for solutions of homogeneous case ($f = 0$). Note that the denominator in \eqref{1.2} is strictly positive. Indeed,
$$
p\left[1-\left(\frac{n}{pq}+\frac{1}{r}\right)\right]-\left[1-\left(\frac{n}{q}+\frac{2}{r}\right)\right] =  (p-1)\left(1-\frac{1}{r}\right) +\frac{1}{r}>0.
$$

Such a quantitative estimate in \eqref{1.2} constituted a long-standing open issue and it was solved up to now only in the linear setting ($p=2$). Notwithstanding, the analyse for the case $p\neq2$ is considerably more challenging.

In our approach, we makes use of an adjusted $\theta-$intrinsic scaling technique inspired by \cite{TU14} and \cite{U08} (see also \cite{DOS18}), where $\theta>0$ is the \textit{intrinsic scaling factor} for the temporal variable, which depends on the magnitude of the gradient inside the critical zone in the following way:
\begin{equation}\label{1.3}
\theta:=2+(2-p)\log_{\rho}(\rho^\alpha+|\nabla u(0,0)|),
\end{equation}
where $\rho\in\left(0,\frac{1}{2}\right)$ is a universal constant (see Lemma \ref{l3.2} for details). Furthermore, when $p\geq2$ our result remains true outside of the critical zone as well (see Section \ref{Sec3.2} for details).

Our estimates are natural extensions (regarding to $C^{1+\alpha}$ scenario) of those obtained in \cite[Theorem 3.4]{TU14} and \cite[Theorem 4.1]{T13} concerning of $C^{\alpha}$ range, and to some extent, of those from \cite{AZ15}, \cite[Theorem 5.5 and 5.9]{APR}, \cite[Section 5]{DT17}, \cite[Theorem 3]{T14}. The novelty of our approach consists of removing the restriction of analyzing the desired $C^{1+\alpha}$ regularity estimates just along the \textit{a priori} unknown set of critical points of solutions (i.e. along the set where gradient vanishes), where the diffusivity of the equation collapses (see, for example, \cite{daSLR16, DOS18, daSRS18, daSS18, T14} and \cite{T15}, where sharp and improved regularity estimates are obtained along the set of certain degenerate points of solutions). Furthermore, unlike \cite{DOS18}, \cite{DT17}, \cite{TU14}, we also treat the singular case, i.e., when $\max\left\{1,\frac{2n}{n+2}\right\}<p<2$, which is a non-trivial task, because in this setting the degeneracy degree of $p$-Laplacian blows-up along critical points.

Heuristically, in order to obtain the desired $C^{1+\alpha}$ regularity estimate, one should approach solutions by suitable affine functions. However, in a specific iterative scheme, these functions ``are not in the kernel of operator'', which provides an accumulative error in each step of approximation (see section \ref{s3} for details, also compare with \cite{ART15}, \cite[Section 4 and 5]{DT17} and \cite{T14}). To overcome this obstacle, we use a technique based on the notion of geometric tangential analysis. More precisely, by using a suitable approximation argument, we show that for each inhomogeneous equation with flattened source term there exists a fine tangential profile, which connects to weak solutions with a small prescribed approximation error - transporting the regularity back to the original equation (Lemma \ref{l2.1} and \ref{l3.1}). We then iterate this estimate in a systematic manner, properly adjusted to the intrinsic scaling of equation. Inspired by the recent results from \cite{ATU17, ATU18, APR} and \cite{TU14}, we obtain an estimate (Theorem \ref{t3.2}), which provides a precise control of the oscillation of weak solution of \eqref{1.1} in terms of the magnitude of its gradient.

Another fundamental aspect in our approach concerns to the geometry of the parabolic cylinders for which our geometric estimates hold. Unlike \cite{TU14} (see also \cite{DD18}  and \cite{DOS18}), we must adjust such cylinders according to range of $p$. Summarily, for the singular case, we consider the standard parabolic cylinder with the $\theta$-intrinsic geometry. On the other hand, for the degenerate counterpart, we must correct the geometry of corresponding cylinders by adjusting (in a suitable manner) its height in order to access the desired regularity estimate via an iterative proceeding (see Section \ref{s3} for a complete explanation about such a construction).

It is worth mentioning, that under appropriate structural conditions, the techniques used in this paper allow one to treat the case of more general evolution equations as follows
$$
  u_t-\Div(\mathcal{A}(x,t,\nabla u))=f,
$$
as long as the operator has suitable $p$-Laplacian structure, which gives access to existence and regularity theory to weak solutions (see, \cite[Ch.2, \S 1]{D93}).

Observe that our result is in accordance with well-known estimates obtained in some specific cases. For example, when $p=2$ from \eqref{1.2} we obtain
$$
\alpha=1-\Big(\frac{n}{q}+\frac{2}{r}\Big),
$$
which is the optimal $C^{1+\alpha}$ exponent for the inhomogeneous heat equation (as was obtained by energy methods in \cite{K08}, see also \cite{DT17} for the fully nonlinear setting). Another example would be the case of bounded (or else, independent) in time force term, that is, roughly speaking, $r=\infty$, then \eqref{1.2} implies
$$
\alpha=\min\left\{\frac{q-n}{q(p-1)},\,\alpha^-_{\mathrm{H}}\right\},
$$
which is the optimal $C^{1,\alpha}$ regularity exponent for the elliptic equation for $p\geq2$ (see \cite{AZ15}). Moreover, if also $q=\infty$ (i.e. the source term is bounded), then \eqref{1.2} gives
$$
  \alpha=\min\left\{\frac{1}{p-1}, \, \alpha^-_{\mathrm{H}}\right\},
$$
which is the expected optimal regularity in certain scenarios of elliptic equations (see \cite{ATU17}, \cite{ATU18} and \cite{APR}).

The table below provides a global picture (recent advances) between the elliptic (cf. \cite{ATU17}, \cite{ATU18}, \cite{AZ15} and \cite[Theorem 4.1]{T13}) and parabolic regularity theory (cf. \cite[Theorem 3.4]{TU14}, see also \cite{DT17}) for equations of $p-$Laplacian type:
{\scriptsize{
\begin{table}[h]
\centering
\begin{tabular}{|c|c|c|c|c}
\cline{1-4}
 $f \in L^q(B_1)$  & \text{Sharp Regularity} & $f \in L^{q, r}(Q_1)$ & \text{Sharp Regularity}  \\ \cline{1-4}
   $\frac{n}{p} < q<n$ &  $C_{\loc}^{0, \min\{\alpha_0^{-}, \beta(n, p, q)\}}$ & $\frac{n}{q}+\frac{2}{r}>1>\frac{1}{r}+\frac{n}{pq}$ & $\text{par}-C_{\loc}^{\alpha}, \,\,p>2$ \\ \cline{1-4}
   $q=n$ &  \text{Open problem} & $\frac{n}{q}+\frac{2}{r}=1>\frac{1}{r}+\frac{n}{pq}$ & \text{Open problem} \\ \cline{1-4}
   $n<q< \infty$ &  $C_{\loc}^{1,\min\left\{\alpha_{\mathrm{H}}^{-}, \frac{q-n}{q(p-1)}\right\}}, \,\,p>2$ & Condition (C)& $\text{par}-C_{\loc}^{1+\alpha}$  \\ \cline{1-4}
   $q=\infty$ &  $C_{\loc}^{1,\min\left\{\alpha_{\mathrm{H}}^{-}, \frac{1}{p-1}\right\}}. \,\,p>2$ & $q=r=\infty$ & $C_{\loc}^{1,\min\left\{\alpha_{\mathrm{H}}^{-}, \frac{1}{p-1}\right\}}$  \\ \cline{1-4}
\multicolumn{4}{|l|}{ \hspace{1.3cm} {\bf Elliptic Theory  \qquad \qquad \qquad  \qquad \quad \quad  Parabolic Theory}} &  \\ \cline{1-4}
\end{tabular}
\end{table}}}

Up to date, the sharp regularity on the borderline conditions $q=n$ and $q=\infty$ (resp. $\frac{n}{q}+\frac{2}{r}=1$ and $q=r=\infty$) were established (in general) just for some special sceneries, in particular for the linear case and for $p$-Poisson equation in $2-$D, see \cite[Theorem 1]{ATU17}, \cite[Theorem 2]{ATU18}, \cite{DT17} and \cite{K08} for details.

The paper is organized as follows: in Section \ref{s2} we show that solutions of \eqref{1.1} can be approximated by suitable $p$-caloric functions. In Section \ref{s3} we prove the main result of this paper (Theorem \ref{t3.1}) by obtaining sharp estimates inside  and outside of the critical zone. As a consequence, we find precisely how the modulus of continuity degenerates (improves) along certain $\varepsilon$-layers related to borderline conditions in \eqref{cc} (Corollary \ref{c3.2} and Corollary \ref{c3.3}).

\section{Preliminaries}\label{s2}

In this section we prove that weak solutions of \eqref{1.1} can be approximated by $p$-caloric functions. We start by the notion of weak solutions.

\begin{definition}\label{d2.1} A function $u \in  C_{\loc}(0, T; L^2(\Omega)) \cap L^p_{\loc}(
0, T; W_{\loc}^{1, p}(\Omega))$ is called a weak solution of \eqref{1.1} in $\Omega\times(0,T]$, if for every compact set $K\subset\Omega$, every $[t_1, t_2] \subset (0, T]$ and $\psi\in H^1_{\loc}(0, T; L^2(K))\cap L^p_{\loc}(0,T; W_{\loc}^{1, p}(K))$ there holds
$$
\displaystyle \left.\int_{K} u \psi\,dx \right|^{t_2}_{t_1} + \int_{t_1}^{t_2} \int_{K} \left[-u \psi_t+ |\nabla u|^{p-2}\nabla u \cdot \nabla \psi \right]\,dx\,dt = \int_{t_1}^{t_2} \int_{K} f \psi\,dx\,dt.
$$
\end{definition}

An equivalent definition of weak solutions via Steklov average allows to prove the next Caccioppoli type estimate, which plays an essential role in proving the existence of $p$-caloric approximation of weak solutions.

\begin{proposition}\label{p2.1} Let $K \times [t_1, t_2] \subset \Omega \times (0, T]$. If $u$ is a weak solution of \eqref{1.1}, then there exists a constant $C>0$, depending only on $n$, $p$ and $K\times [t_1, t_2]$ such that
\begin{eqnarray*}
&&\sup\limits_{t_1<t<t_2}\int_{K} u^2 \xi^p\,dx + \int_{t_1}^{t_2} \int_{K} |\nabla u|^p \xi^p\,dx\,dt\nonumber\\
&\leq&\int_{t_1}^{t_2} \int_{K} |u|^p(\xi^p+|\nabla \xi|^p)\,dx\,dt+ C\int_{t_1}^{t_2}\int_{K} u^2\xi^{p-1}|\xi_t|\,dx\,dt + C\|f\|_{L^{q, r}},
\end{eqnarray*}
for every $\xi \in C_0^{\infty}(K \times (t_1, t_2); [0, 1])$.
\end{proposition}
\begin{proof}
We sketch the proof here. It follows by taking $\psi=u_h\xi^p$ as a test function, where $u_h$ is Steklov average of $u$, i.e.
$$
u_h:=\left\{
\begin{array}{rcl}
\displaystyle\frac{1}{h}\int_{t}^{t+h} u(\cdot, \tau)\,d\tau  & \mbox{if} & t \in (0, T-h],\\
0 & \mbox{if} & t \in (T-h, T].
\end{array}
\right.
$$
With the usual combination of integrating in time, passing to the limit as $h\rightarrow0$ and applying Young's inequality, we get the desired estimate. We refer the reader to \cite[Ch.3, \S 6]{dBGV12} for details.
\end{proof}
In order to state the next result, we set
$$
Q_{\rho}(x_0, t_0):=B_{\rho} \times(t_0-\rho^{\theta}, t_0],
$$
where $B_\rho$ is the ball of radius $\rho>0$ centered at the origin. Note that
$$
Q_1:=Q_1(0,0)=B_1\times(-1,0].
$$

\begin{lemma}\label{l2.1}
If $u$ is a weak solution of \eqref{1.1} in $Q_1$ with
$\|u\|_{L^{\infty}(Q_1)} \leq 1$, then $\forall\varepsilon>0$ there exists $\delta= \delta(p,n,\varepsilon)>0$ such that whenever $\|f\|_{L^{q, r}(Q_1)}\leq\delta$ there exists a $p$-caloric function $\phi: Q_{1/2} \to \R$ such that
\begin{equation}\label{2.1}
\max\left\{\|u-\phi\|_{L^{\infty}(Q_{1/2})},\,\|\nabla(u-\phi)\|_{L^{\infty}(Q_{1/2})}\right\}<\varepsilon.
\end{equation}
\end{lemma}
\begin{proof}
We argue by contradiction. Thus, for an $\varepsilon_0>0$ there is a sequence
$$
u_k\in C_{\loc}(-1, 0; L^2(B_1)) \cap L^p_{\loc}(-1, 0; W_{\loc}^{1, p}(B_1))
$$
and a sequence
$$
f_k\in L^{q, r}(Q_1)
$$
such that
\begin{equation}\label{2.2}
(u_k)_t - \Div(|\nabla u_k|^{p-2}\nabla u_k)  = f_k\quad \mbox{in} \quad Q_1
\end{equation}
with
\begin{equation}\label{2.3}
\|u_k\|_{L^{\infty}(Q_1)} \le 1 \ \ \mbox{and} \ \ \|f_k\|_{L^{q,r}(Q_1)} = o(1) \ \ \mbox{as} \ \  k \rightarrow \infty,
\end{equation}
and at the same time
\begin{equation}\label{2.4}
  \text{either} \quad \|u_k-\phi\|_{L^{\infty}(Q_{1/2})}>\varepsilon_0 \quad  \mbox{or}\quad\|\nabla(u_k-\phi)\|_{L^{\infty}(Q_{1/2})}>\varepsilon_0,
\end{equation}
for any $p$-caloric function in $Q_{1/2}$, i.e.
$$
\phi_t-\Div(|\nabla\phi|^{p-2}\nabla\phi)=0\,\,\textrm{ in }\,\,Q_{1/2}.
$$
We now fix a cutoff function $\xi \in C_0^{\infty}(Q_1, [0, 1])$ such that $\xi \equiv 1$ in $Q_{1/2}$ and $\xi \equiv 0$ on $Q_{9/10}^c$. From Proposition \ref{p2.1} we obtain
$$
\begin{array}{rcl}
 \|u_k\|_{V(Q_{1/2})} & \leq &\displaystyle \sup\limits_{-1<t<0}\int_{B_1} u_k^2\xi^p + \int_{-1}^{0} \int_{B_1} |\nabla u_k|^p \xi^p\\
  & \leq & \displaystyle \int_{-1}^{0} \int_{B_1} |u_k|^p(\xi^p+|\nabla \xi|^p)+ \int_{-1}^{0} \int_{B_1}u_k^2\xi^{p-1}|\xi_t|  + \|f_k\|_{L^{q, r}}\\
   & \leq & C(\xi,p,n) + \text{o}(1) \quad \text{as} \quad k \to \infty,
\end{array}
$$
where
$$
V(\Omega \times I):= L^{\infty}\left(I; L^2(\Omega)\right) \cap L^p\left(I;W^{1, p}(\Omega)\right),
$$
and $C>0$ is a constant depending only on $\xi$, $p$ and dimension. Therefore, up to a subsequence $u_k$ converges weakly in $V\left(Q_{1/2}\right)$. Hence,
$$
\|(u_k)_t\|_{L^{\ell, 1}(Q_{1/2})}\leq C,
$$
where $\ell:=\frac{p}{p-1}<p$, $p\geq2$ (see \cite{L08}), and
$$
\|(u_k)_t\|_{L^{2}(Q_{1/2})} \leq C,
$$
when $\max\left\{1,\frac{2n}{n+2}\right\}<p<2$ (see \cite{AMS04}).
Making use of the embedding
$$
  W^{1, p} \hookrightarrow L^p \subset L^{\ell},
$$
whenever $p\geq2$ and
$$
L^2 \hookrightarrow L^{p},
$$
whenever $\max\left\{1,\frac{2n}{n+2}\right\}<p<2$, we deduce (see \cite[Corollary 4]{S87}) that
\begin{equation}\label{2.5}
u_k\rightarrow\phi\,\,\textrm{ strongly in }\,\,L^p(Q_{1/2}),
\end{equation}
for a function $\phi$. Using \eqref{2.3} and \eqref{2.5} and passing to the limit in \eqref{2.2}, one concludes that $\phi$ is a $p$-caloric function. Moreover, since $u_k$ is a bounded weak solution of \eqref{2.2}, then its spacial gradient is locally H\"{o}lder continuous (see \cite{BC04, C91, D93, DF84, DF85, DF285, DiBUV02, LSU68, L96, W86}), implying that the convergence $u_k\rightarrow\phi$ is locally uniform in $C^1$. Therefore,
$$
\max\left\{\|u_k-\phi\|_{L^{\infty}(Q_{1/2})},\, \|\nabla(u_k-\phi)\|_{L^{\infty}(Q_{1/2})}\right\} \to 0 \quad \text{as} \quad k \to \infty,
$$
which contradicts to \eqref{2.4}.
\end{proof}

\begin{remark}\label{r2.1}
Note that if $u$ is any weak solution of \eqref{1.1} in $Q_1$, then it is possible to normalize it in such a way, that the normalized function satisfies conditions of Lemma \ref{l2.1}. More precisely, for a $\delta>0$ and $s>0$ fixed, there exists positive constant $\mu=\mu(\delta,s)$ such that the function
$$
v(x,t):= \mu^{s}u(\mu^{s}x,\mu^{\tau}t),
$$
satisfies the conditions of Lemma \ref{2.1}, where $\tau:=2s(p-1)>0$.
\end{remark}

To see this, observe that
$$
v_t-\Delta_p v  = \mu^{s+\tau}u_t(\mu^{s}x,\mu^{\tau}t)- \mu^{(2p-1)s}(\Delta_p u)(\mu^{s}x,\mu^{\tau}t).
$$
Since $\tau=2s(p-1)$, then
$$
\tau+s=(2p-1)s.
$$
On the other hand, $v$ is a weak solution of
$$
v_t-\Div(|\nabla v|^{p-2}\nabla v) = \mu^{(2p-1)s}f(\mu^{s}x,\mu^{(2p-1)s}t)=:g(x,t).
$$
We estimate
$$
\|g\|_{L^{q,r}(Q_1)} \le\mu^{(2p-1)s-\left(\frac{sn}{q}+\frac{(2p-1)s}{r}\right)}\|f\|_{L^{q,r}(Q_{\mu})}.
$$
Set
$$
\begin{array}{rcl}
  \kappa & := & (2p-1)s-\left(\frac{sn}{q}+\frac{(2p-1)s}{r}\right) \\
   & = & s\left[(p-1)\left(1-\frac{1}{r}\right)+\frac{1}{r}\right] + sp\left[1-\left(\frac{n}{pq}+\frac{1}{r}\right)\right].
\end{array}
$$
Using the minimal integrability condition from \eqref{cc}, we point out that $\kappa>0$ for any $s>0$. Therefore, for $\delta>0$ fixed, if we choose
$$
0<\mu<\min\left\{ 1, \left(\frac{1}{\|u\|_{L^{\infty}(Q_1)}}\right)^{\frac{1}{s}},\,\left(\frac{\delta}{\|f\|_{L^{q, r}(Q_1)}}\right)^{\frac{1}{\kappa}}\right\},
$$
then $\|v\|_{L^{\infty}(Q_1)} \leq 1$ and $\|g\|_{L^{q,r}(Q_1)}\le\delta$.

Next, we analyse some aspects of scaling factor. Set
$$
\hat{\alpha} :=\frac{1-\left(\frac{n}{q}+\frac{2}{r}\right)}{p\left[1-\left(\frac{n}{pq}+\frac{1}{r}\right)\right]-\left[1-\left(\frac{n}{q}+\frac{2}{r}\right)\right]}.
$$
Observe that $\alpha \leq \hat{\alpha}$, where is defined by \eqref{1.2}. Now, as in \eqref{1.3} we consider
$$
\theta(\alpha) = 2+(2-p)\log_{\lambda}(\lambda^{\alpha} +|\nabla u(0,0)|).
$$
By assuming that $|\nabla u(0, 0)| \leq \lambda^{\alpha}$ for $\lambda \ll 1$ we have
$$
\lambda^{\hat{\alpha}} \leq \lambda^{\hat{\alpha}} + |\nabla u(0, 0)| \leq \lambda^{\alpha} + |\nabla u(0, 0)|\leq 1.
$$
Hence,
$$
\hat{\alpha} \geq \log_{\lambda}(\lambda^{\hat{\alpha}} +|\nabla u(0, 0)|) \geq \log_{\lambda}(\lambda^{\alpha} +|\nabla u(0, 0)|) \geq 0.
$$
For $\max\left\{1, \frac{2n}{n+2}\right\}<p\leq 2$ we have the following
$$
2\leq \theta(\alpha) \leq  2+(2-p)\hat{\alpha}.
$$
On the other hand, if $p>2$ we obtain
$$
2+(2-p)\hat{\alpha} \leq \theta(\alpha) \leq 2.
$$
Therefore,
$$
\min\{2, 2+(2-p)\hat{\alpha}\} \leq \theta(\alpha) \leq \max\{2, 2+(2-p)\hat{\alpha}\}.
$$
Moreover, it is easy to check (for $p> 2$) that
$$
2+(2-p)\hat{\alpha} = \frac{1+\frac{2}{p-2}+\frac{n}{q}}{1-\frac{1}{r}+\frac{1}{p-2}} \in (1, 2).
$$
In other words, for $p>2$
$$
1< \frac{1+\frac{2}{p-2}+\frac{n}{q}}{1-\frac{1}{r}+\frac{1}{p-2}} \leq \theta(\alpha) \leq 2.
$$
Finally, for $\max\left\{1, \frac{2n}{n+2}\right\}<p\leq 2$ we have
$$
2\leq 2+(2-p)\hat{\alpha} \leq 3.
$$
Consequently, for $\max\left\{1, \frac{2n}{n+2}\right\}<p\leq 2$
$$
2 \leq \theta(\alpha)\leq 3.
$$

\section{Sharp regularity estimates}\label{s3}

In this section we prove the main result of the paper. Let $S_\rho^\alpha$ be the critical zone of solutions, i.e.
$$
S_{\rho}^{\alpha}(Q_1):=\left\{(x, t) \in Q_1;\,|\nabla u(x, t)|\leq\rho^{\alpha}\right\},
$$
where $\rho>0$ is small, and $\alpha$ is defined by \eqref{1.2}. In order to proceed, we define corrected parabolic cylinder
$$
  \hat{Q}_{\rho^k}(x_0, t_0) := B_{\rho} \times \left(t_0-\rho^{\theta(\sigma + k-1)}, t_0\right],
$$
where $k \in \mathbb{N}$ and
$$
\sigma := \min\left\{1, \frac{2}{2+(2-p)\hat{\alpha}}\right\}.
$$
\begin{remark}\label{r3.1}
If $\max\left\{1, \frac{2n}{n+2}\right\}<p\leq 2$, then
$$
\hat{Q}_{\rho}(x_0, t_0)= Q_{\rho}(x_0, t_0).
$$
If $p>2$, then $\hat{Q}_{\rho}(x_0, t_0) \subset Q_{\rho}(x_0, t_0)$. Also, for any $\max\left\{1, \frac{2n}{n+2}\right\}<p< \infty$ one has $\sigma\theta\geq 2$.
\end{remark}
The following theorem is the main result of this paper. We will prove it by analysing sharp estimates inside and outside of the critical zone.
\begin{theorem}\label{t3.1}
Let $K\subset\subset Q_1$, $u$ be a bounded weak solution of \eqref{1.1} in $Q_1$ and let \eqref{cc} hold. If $(x_0, t_0) \in S_\rho^\alpha(K)$, $\rho>0$, then $u$ is $C^{1+\alpha}$ (in the parabolic) at $(x_0, t_0)$, i.e., there exists a constant $M>0$ such that
$$
\displaystyle \sup_{\hat{Q}_{\rho}(x_0, t_0) \cap K} \left|u(x,t)-u(x_0,t_0)-\nabla u(x_0,t_0)\cdot (x-x_0)\right|\leq M\rho^{1+\alpha},
$$
for $\rho>0$ small enough, and $\alpha$ is defined by \eqref{1.2}. Moreover, if $p>2$ and $(x_0,t_0)\notin S_\rho^\alpha(K\cap Q_1)$, then the conclusion of the theorem is still true.
\end{theorem}

\subsection{Sharp estimates in the critical zone}

Using an iterative argument, we prove the desired regularity estimate in the critical zone.
The estimate in Lemma \ref{l2.1} can be further improved up to the sharp exponent in our compatibility regime \eqref{cc}. The following lemma serves that purpose and provides the first step of such iteration.
\begin{lemma}\label{l3.1}
Let $u$ be a weak solution of \eqref{1.1} in $Q_1$ with $\|u\|_{L^\infty(Q_1)}\leq1$. There exist $\delta>0$ and $\lambda\in \left(0,\frac{1}{2}\right)$ such that if $\|f\|_{L^{q, r}(Q_1)}\leq\delta$, then
$$
\displaystyle\sup_{\hat{Q}_{\lambda}} \left|u(x, t)-u(0, 0)-\nabla u(0,0)\cdot x\right|\leq \lambda^{1+\alpha},
$$
where $\alpha$ is defined by \eqref{1.2}.
\end{lemma}
\begin{proof}
Let $\varepsilon>0$. From Lemma \ref{l2.1} we know that there exists a $p$-caloric function $\phi$ and $\delta>0$, such that whenever $\|f\|_{L^{q, r}(Q_1)}\leq\delta$, then \eqref{2.1} holds. Taking $\lambda\in \left(0,\frac{1}{2}\right)$ (to be chosen \textit{a posteriori}) and $(x,t)\in \hat{Q}_\lambda$, and from H\"{o}lder gradient continuity (cf. \cite{BC04, C91, D93, DF84, DF85, DF285, DiBUV02, LSU68, L96, W86}) we then estimate
\begin{eqnarray}\label{3.1}
&&\displaystyle \sup_{\hat{Q}_{\lambda}} \left|u(x, t)-u(0, 0)-\nabla u(0, 0)\cdot x\right|\nonumber\\
&\leq& \displaystyle\sup_{\hat{Q}_{\lambda}}\left|\phi(x,t)-\phi(0,0)-\nabla\phi(0, 0)\cdot x\right|\nonumber \\
&+& \displaystyle\sup_{\hat{Q}_{\lambda}} |(u-\phi)(x,t)|+|(u-\phi)(0,0)|+|\nabla(u-\phi)(0,0)|\nonumber \\
&\leq & \displaystyle C\sup_{\hat{Q}_{\lambda}} \left(|x|+ \sqrt{|t|}\right)^{1+\alpha_H} + 3\varepsilon \nonumber \\
&\leq& C\lambda^{(1+\alpha_H)\min\left\{1, \frac{\theta\sigma}{2}\right\}}+3\varepsilon\nonumber\\
&\leq & \displaystyle C \lambda^{1+\alpha_H} + 3\varepsilon,
\end{eqnarray}
where $C>0$ is a universal constant. By fixing
\begin{equation}\label{3.2}
\lambda\in \left(0, \min\left\{\frac{1}{2}, \left(\frac{1}{2 C}\right)^{\frac{1}{\alpha_H-\alpha}}\right\}\right),
\end{equation}
and choosing $\varepsilon\in \left(0,\frac{1}{6} \lambda^{1+\alpha}\right)$ in \eqref{3.1}, we conclude the proof.
\end{proof}

Note that the previous lemma is not enough to proceed with an iterative scheme, because a priori we do not know the equation which is satisfied by
$$
 \frac{(u-L_k)(\lambda^kx, \lambda^{k\theta}t)}{\lambda^{k(1+\alpha)}},
$$
where $\{L_k\}_{k\in\mathbb{N}}$ is sequence of affine functions (compare with \cite{ART15}, \cite[Theorem 2]{DD18} and \cite[Section 5 and 6]{DT17} in the fully nonlinear setting). Nevertheless, it provides  the following information on the oscillation of $u$ in $\hat{Q}_{\lambda}$.

\begin{corollary}\label{c3.1}
Under the conditions of Lemma \ref{3.1} one has
$$
\displaystyle \sup_{\hat{Q}_{\lambda}}\left|u(x,t)-u(0,0)\right|\leq\lambda^{1+\alpha}+\lambda|\nabla u(0,0)|,
$$
where $\lambda$ is a constant satisfying \eqref{3.2}, and $\alpha$ is defined by \eqref{1.2}.
\end{corollary}
\begin{proof}
Using Lemma \ref{l3.1} we estimate
{\small{
$$
\begin{array}{rcl}
  \displaystyle \sup_{\hat{Q}_{\lambda}} \left|u(x,t)-u(0,0)\right| & \leq &  \displaystyle \sup_{\hat{Q}_{\lambda}} \left|u(x,t)-u(0,0)-\nabla u(0,0)\cdot x\right|+\sup_{\hat{Q}_{\lambda}}|\nabla u(0,0)\cdot x| \\
   & \leq & \lambda^{1+\alpha}+\lambda|\nabla u(0, 0)|.
\end{array}
$$}}
\end{proof}

In order to obtain a precise control on the influence of magnitude of the gradient of $u$, we iterate solutions (using Corollary \ref{c3.1}) in corrected $\lambda$-adic cylinders.

\begin{lemma}\label{l3.2} Under the assumptions of Lemma \ref{l3.1} one has
\begin{equation}\label{3.3}
\displaystyle\sup_{\hat{Q}_{\lambda^k}}\left|u(x,t)-u(0,0)\right|\leq\lambda^{k(1+\alpha)}+|\nabla u (0,0)|\sum_{j=0}^{k-1}\lambda^{k+j\alpha},
\end{equation}
where $\lambda$ is a constant satisfying \eqref{3.2} and $\alpha$ is defined by \eqref{1.2}.
\end{lemma}
\begin{proof}
We argue by induction. When $k=1$, we have \eqref{3.3} from Lemma \ref{l3.1}. Suppose now that \eqref{3.3} holds for all the values of $l=1,2,\cdots,k$. Our aim is to prove it for $l=k+1$. Define $v_k: \hat{Q}_1:= B_1 \times (-\lambda^{\theta(\sigma-1)}, 0]$ given by
$$
\displaystyle v_k(x,t):=\frac{u(\lambda^k x, \lambda^{k\theta}t)-u(0,0)}{\lambda^{k(1+\alpha)}+|\nabla u(0,0)|\sum\limits_{j=0}^{k-1}\lambda^{k+j\alpha}},
$$
where $\theta>0$ is a constant to be chosen later. By induction hypothesis $\|v_k\|_{L^{\infty}(\hat{Q}_1)}\le 1$. Note also that $v_k(0,0)=0$,
$$
(v_k)_t(x,t)=\frac{\lambda^{k\theta} u_t(\lambda^kx, \lambda^{k\theta} t)}{\lambda^{k(1+\alpha)}+|\nabla u(0,0)|\sum\limits_{j=0}^{k-1}\lambda^{k+j\alpha}},
$$
$$
\nabla v_k(x,t)=\frac{\lambda^{k} \nabla u(\lambda^kx, \lambda^{k\theta} t)}{\lambda^{k(1+\alpha)}+|\nabla u (0,0)|\sum\limits_{j=0}^{k-1}\lambda^{k+j\alpha}}
$$
and
$$
\Delta_p v_k(x,t) = \frac{\lambda^{kp}(\Delta_p u)(\lambda^{k}x,\lambda^{k\theta} t)}{\left(\lambda^{k(1+\alpha)}+|\nabla u(0,0)|\sum\limits_{j=0}^{k-1}\lambda^{k+j\alpha}\right)^{p-1}}.
$$
Choosing
$$
\theta:=2+\log_{\lambda^k}\left(\lambda^{k\alpha} + |\nabla u(0,0)|\sum\limits_{j=0}^{k-1}\lambda^{j\alpha}\right)^{2-p},
$$
we obtain
$$
(v_k)_t-\Delta_p v_k=\frac{\lambda^{kp}f(\lambda^kx, \lambda^{k\theta}t)}{ \left(\lambda^{k(1+\alpha)}+|\nabla u (0,0)|\sum\limits_{j=0}^{k-1}\lambda^{k+j\alpha} \right)^{p-1}}=:f_k(x, t).
$$
We then estimate
$$
\begin{array}{rcl}
\|f_k\|_{L^{q, r}(\hat{Q}_1)}&=&\displaystyle\left(\int_{-\lambda^{\theta(\sigma-1)}}^{0}\left(\int_{B_1}|f_k(x, t)|^q\,dx\right)^{\frac{r}{q}}\,dt\right)^{\frac{1}{r}}\\
&= &\frac{\lambda^{k\left[1-\left(\frac{n}{q}+\frac{\theta}{r}\right)\right]}}{ \left(\lambda^{k\alpha}+|\nabla u (0,0)|\sum\limits_{j=0}^{k-1}\lambda^{j\alpha} \right)^{p-1}}\|f\|_{L^{q, r}\left(\hat{Q}_{\lambda^k}\right)} \\
&=&\frac{\lambda^{k\left[1-\left(\frac{n}{q}+\frac{2}{r}\right)\right]}}{ \left(\lambda^{k\alpha}+|\nabla u(0,0)|\sum\limits_{j=0}^{k-1}\lambda^{j\alpha} \right)^{p\left[1-\left(\frac{n}{pq}+\frac{1}{r}\right)\right]-\left[1-\left(\frac{n}{q}+\frac{2}{r}\right)\right]}}\|f\|_{L^{q, r}\left(\hat{Q}_{\lambda^k}\right)} \\
&\leq&\lambda^{k\left\{\left[1-\left(\frac{n}{q}+\frac{2}{r}\right)\right]-\alpha\left\{p\left[1-\left(\frac{n}{pq}+\frac{1}{r}\right)\right]-\left[1-\left(\frac{n}{q}+\frac{2}{r}\right)\right]\right\}\right\}}\|f\|_{L^{q, r}\left(Q_{1}\right)}\\
&\leq&\delta_0.
\end{array}
$$
Hence, one can apply Lemma \ref{l3.1} to $v_k$ and obtain
$$
\displaystyle \sup_{\hat{Q}_{\lambda}} \left|v_k(x, t)-v_k(0,0)\right|\leq\lambda^{1+\alpha}+\lambda|\nabla v_k(0,0)|,
$$
or else
$$
\displaystyle\sup_{\hat{Q}_{\lambda}}\frac{|u(\lambda^k x, \lambda^{k\theta} t)-u(0, 0)|}{\lambda^{k(1+\alpha)}+|\nabla u (0,0)|\sum\limits_{j=0}^{k-1}\lambda^{k+j\alpha}}\leq \lambda^{1+\alpha}+\frac{\lambda^{k+1}|\nabla u(0,0)|}{\lambda^{k(1+\alpha)}+|\nabla u(0,0)|\sum\limits_{j=0}^{k-1}\lambda^{k+j\alpha}},
$$
which, by scaling back provides
$$
  \displaystyle\sup_{\hat{Q}_{\lambda^{k+1}}}|u(x, t)-u(0,0)| \leq  \lambda^{(k+1)(1+\alpha)}+ |\nabla u(0,0)|\sum\limits_{j=0}^{k}\lambda^{k+1+j\alpha}.
$$
The latter is \eqref{3.3} for $k+1$.
\end{proof}
The next result leads to a sharp regularity estimate in the critical zone.
\begin{theorem}\label{t3.2}
Under the assumptions of Lemma \ref{l3.1} there exists a universal constant $M>1$ such that
$$
\displaystyle \sup_{\hat{Q}_{\rho}}|u(x,t)-u(0,0)|\leq M{\rho}^{1+\alpha}\left(1+|\nabla u(0,0)|\rho^{-\alpha}\right),\,\,\forall\rho\in(0,\lambda),
$$
where $\lambda$ is a constant satisfying \eqref{3.2}, and $\alpha$ is defined by \eqref{1.2}.
\end{theorem}
\begin{proof}
Take $\rho\in(0,\lambda)$ and choose $k\in\mathbb{N}$ such that $\lambda^{k+1}<\rho\leq\lambda^k$. Using Lemma \ref{l3.2}, we estimate
\begin{eqnarray*}
\displaystyle\sup_{\hat{Q}_{\rho}}\frac{|u(x,t)-u(0,0)|}{\rho^{1+\alpha}}&\leq&\displaystyle \frac{1}{\lambda^{1+\alpha}}\sup_{\hat{Q}_{\lambda^k}}\frac{|u(x,t)-u(0,0)|}{\lambda^{k(1+\alpha)}}\\
&\leq&\frac{1}{\lambda^{1+\alpha}}\left[1+|\nabla u(0,0)|\frac{\sum\limits_{j=0}^{k-1}\lambda^{k+j\alpha}}{\lambda^{k(1+\alpha)}}\right] \\
&=&\frac{1}{\lambda^{1+\alpha}}\left[1+|\nabla u(0,0)|\lambda^{-k\alpha}\sum\limits_{j=0}^{k-1}\lambda^{j\alpha}\right]\\
&\leq&\frac{1}{\lambda^{1+\alpha}}\left(1+\frac{1}{1-\lambda^{\alpha}}\right)\left(1 +|\nabla u(0,0)|\lambda^{-k\alpha}\right)\\
&\leq&\frac{1}{\lambda^{1+\alpha}}\left(1+\frac{1}{1-\lambda^{\alpha}}\right)\left(1 +|\nabla u(0,0)|\rho^{-\alpha}\right),
\end{eqnarray*}
which concludes the proof.
\end{proof}

As a consequence, we obtain the first part of Theorem \ref{3.1}.

\begin{proof}[Proof of the first part of Theorem \ref{3.1}.]
Without loss of generality, we may assume that $K=Q_{\frac{1}{2}}$ and $(x_0,t_0)=(0,0)$. Using Theorem \ref{t3.2} (re-scaled according to Remark \ref{r2.1}, if needed), we estimate
\begin{eqnarray*}
\displaystyle\sup_{\hat{Q}_\rho}\left|u(x,t)-u(0, 0)-\nabla u(0,0)\cdot x\right| & \leq & \displaystyle\sup_{\hat{Q}_\rho}|u(x,t)-u(0,0)|+|\nabla u(0,0)|\rho \\
&\leq& M\left(1+|\nabla u(0,0)|\rho^{-\alpha}\right)\rho^{1+\alpha}+\rho^{1+ \alpha}\\
&\leq& 3M\rho^{1+\alpha}.
\end{eqnarray*}
\end{proof}

\subsection{Sharp estimates outside of the critical zone}\label{Sec3.2}

Next, we assume $p>2$ and prove the conclusion of Theorem \ref{t3.1}, despite having $(x_0,t_0)\notin S_{\rho}^\alpha(K\cap Q_1)$, i.e., when $|\nabla u(x_0,t_0)|>\rho^\alpha$, thus, completing the proof of Theorem \ref{t3.1}. As before, without loss of generality, we assume that $(x_0,t_0)=(0,0)$. Since $|\nabla u|$ is continuous, one can define $\tau=|\nabla u(0,0)|^{1/\alpha}>0$. Take any $\rho\in(0,\lambda)$. We then analyse all the possible cases.

\textbf{Case 1.} If $\rho\in[\tau,\lambda)$, then from Theorem \ref{t3.2} we obtain
$$
\displaystyle \sup_{\hat{Q}_{\rho}}|u(x,t)-u(0,0)|\leq M{\rho}^{1+\alpha}\left(1+|\nabla u(0,0)|\tau^{-\alpha}\right),
$$
for a constant $M>1$. Hence, $u$ is $C^{1+\alpha}$ (in the parabolic sense) at the origin.

\textbf{Case 2.} If $\rho\in(0,\tau)$, then in order to apply Theorem \ref{t3.2}, we need to properly re-scale $u$. Let
$$
v(x,t):=\frac{u(\tau x,\tau^\gamma t)-u(0,0)}{\tau^{1+\alpha}},
$$
where $\gamma:=2+\alpha(2-p)$. Observe that in $\hat{Q}_1: = B_1 \times \left(-\tau^{\gamma(\sigma-1)}, 0\right]$ the function $v$ is a weak solution of
$$
v_t-\Delta_p v=\tau^{1-\alpha(p-1)}f(\tau x,\tau^\gamma t):=g(x, t).
$$
Note also that
$$
\|g\|_{L^{q, r}(\tilde{Q}_1)} \le\tau^{-\alpha(p-1)+1-\left(\frac{n}{q}+\frac{\gamma}{r}\right)}\|f\|_{L^{q, r}(\hat{Q}_\tau)}.
$$
The choice of $\alpha$ in \eqref{1.2} guarantees that
$$
-\alpha(p-1)+1-\Big(\frac{n}{q}+\frac{\gamma}{r}\Big)\geq0.
$$
Using Theorem \ref{t3.2} and taking into account that $v(0,0)=0$ and $|\nabla v(0,0)|=1$, we get
$$
\displaystyle\sup_{\hat{Q}_1} |v(x,t)|=\sup_{\hat{Q}_{\tau}} \frac{|u(x,t)-u(0,0)|}{\tau^{1+\alpha}}\leq M(1+|\nabla u(0,0)|\tau^{-\alpha}).
$$
On the other hand, $u\in C_{\loc}^{1+\beta}$ for some $\beta\in(0,1)$ (see \cite{BC04, C91, D93, DF84, DF85, DF285, DiBUV02, LSU68, L96, W86}). Hence, there exists $\sigma>0$ small enough such that
$$
|\nabla v(x,t)|>\frac{1}{2},\,\,\,\forall (x,t)\in \hat{Q}_\varsigma.
$$
Therefore, $v$ is a weak solution of a uniformly parabolic equation, i.e.
$$
  v_t-\Div(\mathcal{A}(x,t)\nabla v)=f,
$$
where $\mathcal{A}(x,t)$ is (H\"{o}lder) continuous and $0<a<\mathcal{A}(x,t)<b<\infty$, for some constants $a$ and $b$. Thus, $v\in C^{1+\beta^*}$  locally (see \cite{DT17, K08}), where the sharp exponent is given by
$$
  \beta^*:=1-\Big(\frac{n}{q}+\frac{2}{r}\Big)\geq\alpha.
$$
The last inequality is true, since $p\geq2$. In particular, $v\in C^{1+\alpha}$, so there is a universal constant $C>0$ such that
$$
\displaystyle\sup_{\hat{Q}_{\rho_0}}|v(x,t)-\nabla v(0,0)\cdot x|\leq C\rho_0^{1+\alpha},\,\,\,\forall \rho_0\in\Big(0,\frac{\varsigma}{2}\Big),
$$
that is
$$
\displaystyle\sup_{\hat{Q}_{\rho_0}}\Big|\frac{u(\tau x,\tau^\gamma t)-u(0,0)}{\tau^{1+\alpha}}-\tau^{-\alpha}\nabla u(0,0)\cdot x\Big|\leq C\rho_0^{1+\alpha},
$$
or else
$$
\displaystyle\sup_{\hat{Q}_{\rho_0}}|u(\tau x,\tau^\gamma t)-u(0,0)-\nabla u(0,0)\cdot (\tau x)|\leq C(\tau\rho_0)^{1+\alpha}.
$$
The latter implies for $\rho_0\in\Big(0,\frac{\tau\varsigma}{2}\Big)$
$$
\displaystyle\sup_{\hat{Q}_{\rho_0}}|u(x,t)-u(0,0)-\nabla u(0,0)\cdot x|\leq C\rho_0^{1+\alpha},
$$
which means that $u$ is $C^{1+\alpha}$ at the origin.

Finally, if $\rho_0\in\Big[\frac{\tau\varsigma}{2},\tau\Big)$, then
\begin{eqnarray*}
  \displaystyle\sup_{\hat{Q}_{\rho_0}}|u(x,t)-u(0,0)-\nabla u(0, 0)\cdot x| & \leq & \displaystyle\sup_{\hat{Q}_{\rho_0}}|u(x,t)-u(0,0)|+|\nabla u(0,0)|\tau \\
&\leq & C\tau^{1+\alpha}\\
&\leq & C\left(\frac{2}{\varsigma}\right)^{1+\alpha}\rho_0^{1+\alpha} \\
&=& C\rho_0^{1+\alpha}.
\end{eqnarray*}
Therefore, the desired estimate is true, and the proof of Theorem \ref{t3.1} is complete.

The next result shows precisely how the $C^{1+\alpha}$ modulus of continuity for solutions of \eqref{1.1} degenerates along the $\varepsilon$-levels of $\kappa(n, p, q) = 1-\left(\frac{n}{q}+\frac{2}{r}\right)$ as $\varepsilon$ vanishes.
\begin{corollary}\label{c3.2}
Let $u$ be a bounded weak solution of \eqref{1.1} under the compatibility conditions \eqref{cc}. For any fixed $0<s<1$ and for $\varepsilon \ll 1$, if $f \in L^{\frac{n}{s(1-\varepsilon)}, \frac{2}{(1-s)\varepsilon}}(Q_1)$, then $u$ is $C^{1+\alpha_{\varepsilon}}$ (in the parabolic sense), where
$$
  \alpha_{\varepsilon} = \min\left\{\frac{2\varepsilon}{2(p-1) -(p-2)(1-s)\varepsilon},\alpha_{\mathrm{H}}^{-}\right\} \to 0 \quad \text{as} \quad \varepsilon \to 0+.
$$
\end{corollary}
Finally, we also obtain how the $C^{1+\alpha}$ modulus of continuity for solutions of \eqref{1.1} ``improves asymptotically'' along the $\varepsilon-$levels of
$$
  \varsigma(n, p, q, r) = \frac{nr}{(r-1)q}+\frac{2}{r-1} - (2-p)
$$
as $\varepsilon \to 0$, in the singular scenery, i.e., $\max\left\{1, \frac{2n}{n+2}\right\}<p<2$.

\begin{corollary}\label{c3.3}
Let $u$ be a bounded weak solution of \eqref{1.1} under the compatibility conditions \eqref{cc}. For any $0<s<1$ and for $0<\varepsilon \ll 1$ small enough, if $\max\left\{1, \frac{2n}{n+2}\right\}<p<2$ and $f \in L^{q, r}(Q_1)$, where $q=\frac{nr}{(r-1)[(1-s)(\varepsilon+2-p)]}$ and $r=\frac{2}{s(\varepsilon +2-p)} +1$, then $u$ is $C^{1+\alpha_{\varepsilon}}$ (in the parabolic sense), and
$$
  \alpha_{\varepsilon} \to \alpha_{\mathrm{H}} \quad \text{as} \quad \varepsilon \to 0+.
$$
\end{corollary}
\noindent{\bf Acknowledgements.} The authors would like to thank Eduardo V. Teixeira and Jos\'{e} Miguel Urbano for pointing out several improvements and by their insightful comments and suggestions that benefited a lot the final outcome of this manuscript. The first author thanks department of mathematics of Universidad de Buenos Aires for providing excellent working environment during his visit. J.V. da Silva thanks DM/FCEyN (Universidad de Buenos Aires) for providing a productive working atmosphere. This work was partially supported by Consejo Nacional de Investigaciones Cient\'{i}ficas y T\'{e}cnicas (CONICET-Argentina), Pronex/CNPq/Funcap 00068.01.00/15 and by FCT-Portugal via the grant SFRH/BPD/92717/2013. J.V. da Silva is a member of CONICET.


\begin{thebibliography}{99}

\bibitem{AM07} E. Acerbi, G. Mingioni, \textit{Gradient estimates for a class of parabolic systems}, Duke Math. J. 136 (2007), 285-320.

\bibitem{AMS04} E. Acerbi, G. Mingioni, G.A. Seregin, \textit{Regularity results for parabolic systems relates to a class of non-Newtonian fluids}, Ann. Inst. H. Poincar\'{e} Anal. Non Lin\'{e}aire 21 (2004), 25-60.

\bibitem{AT15} M. Amaral, E.V. Teixeira, \textit{Free transmission problems}, Comm. Math. Phys. 337 (2015), 1465-1489.

\bibitem{ART15} D.J. Ara\'{u}jo, G.C. Ricarte, E.V. Teixeira, \textit{Geometric gradient estimates for solutions to degenerate elliptic equations}. Calc. Var. Partial Differential Equations 53 (2015), 605-625.

\bibitem{ATU17} D.J. Ara\'{u}jo, E.V. Teixeira, J.M. Urbano, \textit{A proof of the $C^{p^{\prime}}$ regularity conjecture in the plane},  Adv. Math. 316 (2017), 541-553.

\bibitem{ATU18} D.J. Ara\'{u}jo, E.V. Teixeira, J.M. Urbano, \textit{Towards the $C^{p^{\prime}}$ regularity conjecture}, to appear in Int. Math. Res. Not., DOI: 10.1093/imrn/rnx068.

\bibitem{AZ15} D.J. Ara\'{u}jo, L. Zhang, \textit{Optimal $C^{1, \alpha}$ estimates for a class of elliptic quasilinear equations}, arXiv:1507.06898.

\bibitem{APR} A. Attouchi, M. Parviainen, E. Ruosteenoja, \textit{$C^{1,\alpha}$ regularity for the normalized $p$-Poisson problem}, J. Math. Pures Appl. (9) 108 (2017), 553-591.

\bibitem{BC04} H.-O. Bae, H.J. Choe, \textit{Regularity for Certain Nonlinear Parabolic Systems}, Comm. Part. Diff. Equations 29 (2004), 611-645.

\bibitem{BDM13} V. B\"{o}gelein, F. Duzaar, G. Minigione, \textit{The regularity of general parabolic systems with degenerate diffusions}, Mem. Amer. Math. Soc. 221 (1041), 2013, 143pp.

\bibitem{C89} L. Caffarelli, \textit{Interior apriori estimates for solutions of fully nonlinear equations},  Ann. of Math. (2) 130 (1989), 189-213.

\bibitem{CLRT14} S. Challal, A. Lyaghfouri, J.F. Rodrigues, R. Teymurazyan, \textit{On the regularity of the free boundary for quasilinear obstacle problems}, Interfaces Free Bound. 16 (2014), 359-394.

\bibitem{C91} H. J. Choe, \textit{H\"{o}lder regularity for the gradient of solutions of certain singular parabolic equations}. Comm. Part. Diff. Equations 16 (1991), 1709-1732.

\bibitem{DD18} J.V. da Silva, D. dos Prazeres, \textit{Schauder type estimates for ``flat'' viscosity solutions to non-convex fully nonlinear parabolic equations and applications}, to appear in Potential Anal., DOI: 10.1007/s11118-017-9677-z.

\bibitem{daSLR16} J.V. da Silva, R.A. Leit\~{a}o, G.C. Ricarte,  \textit{Fully nonlinear elliptic equations of degenerate/singular type with free boundaries}, submitted.\label{daSLR16}

\bibitem{DOS18} J.V. da Silva, P. Ochoa, A. Silva, \textit{Regularity for degenerate evolution equations with strong absorption}, J. Differential Equations 264 (2018), 7270-7293.

\bibitem{daSRS18} J.V. da Silva, J.D. Rossi,  A. Salort, \textit{Regularity properties for $p-$dead core problems and their asymptotic limit as $p \to \infty$}, submitted.\label{daSRS18}

\bibitem{daSS18} J.V. da Silva, A. Salort, \textit{Sharp regularity estimates for quasi-linear elliptic dead core problems and applications}, to appear in Calc. Var. Partial Differential Equations DOI: 10.1007/s00526-018-1344-8.\label{daSS18}

\bibitem{DT17} J.V. da Silva, E.V. Teixeira, \textit{Sharp regularity estimates for second order fully nonlinear parabolic equations}, Math. Ann. 369 (2017), 1623-1648.

\bibitem{D93} E. DiBenedetto, \textit{Degenerate parabolic equations}, Springer-Verlag, New York, 1993, xvi +387 pp.

\bibitem{DF84} E. DiBenedetto, A. Friedman, \textit{ Regularity of solutions of nonlinear degenerate parabolic systems}, J. Reine Angew. Math. 349 (1984), 83-128.

\bibitem{DF85} E. DiBenedetto, A. Friedman, \textit{H\"{o}lder estimates for nonlinear degenerate parabolic systems}, J. Reine Angew. Math. 357 (1985), 1-22.

\bibitem{DF285} E. DiBenedetto, A. Friedman, Addendum to: \textit{H\"{o}lder estimates for nonlinear degenerate parabolic systems}, J. Reine Angew. Math. 363 (1985), 217-220.

\bibitem{dBGV12} E. DiBenedetto, U. Giannazza, V. Vespri, \textit{Harnack's Inequality for Degenerate and Singular Parabolic Equations}, Springer Monographs in Mathematics, Springer, New York, 2012.

\bibitem{DiBUV02} E. DiBenedetto, J.M. Urbano, V. Vespri, \textit{Current issues on singular and degenerate evolution equation}. \textit{ Handbook of Differential Equations} 1, 169-286, 2002.\label{DiBUV}

\bibitem{IM89} T. Iwaniec, J.J. Manfredi, \textit{Regularity of $p-$harmonic functions on the plane}, Rev. Mat. Iberoamericana 5 (1989), 1-19.

\bibitem{KL00} J. Kinnunen, J. Lewis, \textit{Higher integrability for parabolic systems of $p-$Laplacian type}, Duke Math. J. 102 (2000), 253-271.

\bibitem{K08} N. Krylov, \textit{Lectures on elliptic and parabolic equations in Sobolev spaces}, Graduate Studies in Mathematics 96, American Mathematical Society, Providence, RI, 2008, xviii+357pp.

\bibitem{KM11} T. Kuusi, G. Mingioni, \textit{Nonlinear potential estimates in parabolic problems}, Atti. Accad. Naz. Lincei Rend. Lincei Mat. Appl. 22 (2011), 161-174.

\bibitem{KM13} T. Kuusi, G. Mingioni, \textit{Gradient regularity for nonlinear parabolic equations}, Ann. Sc. Norm. Super. Pisa Cl. Sci. (5) 12 (2013), 755-822.

\bibitem{LSU68} O.A. Ladyzhenskaya, V.A. Solonnikov, N.N. Ural'tseva, \textit{Linear and quasi-linear equations of parabolic type}, Translation of Mathematical Monographs 23, American Mathematical Society, Providence, RI, 1968, xi+648pp.

\bibitem{L96} G.M. Lieberman, \textit{Second order parabolic differential equations}, World Scientific Publishing Co., Inc., River Edge, NJ, 1996, xii+439pp.

\bibitem{LL17} E. Lindgren, P. Lindqvist, \textit{Regularity of the $p$-Poisson equation in the plane}, J. Anal. Math. 132 (2017), 217-228.

\bibitem{L08} P. Lindqvist, \textit{On the time derivative in a quasilinear equation}, Skr. K. Nor. Vidensk. Selsk. 2 (2008), 1-7.

\bibitem{S87} J. Simon, \textit{Compact sets in the space $L^p(0,T:B)$}, Ann. Mat. Pura Appl. (4) 146 (1987), 65-96.

\bibitem{T13} E.V. Teixeira, \textit{Sharp regularity for general Poisson equations with borderline sources}, J. Math. Pures Appl. (9) (2013), 150-164.

\bibitem{T14} E.V. Teixeira, \textit{Regularity for quasilinear equations on degenerate singular set}, Math. Ann. 358 (2014), 241-256.

\bibitem{T15} E.V. Teixeira, \textit{Hessian continuity at degenerate points in nonvariational elliptic problems}, Int. Math. Res. Not. 2015 (2015), 6893-6906.

\bibitem{TU14} E.V. Teixeira, J.M. Urbano, \textit{A geometric tangential approach to sharp regularity for degenerate evolution equations}, Anal. PDE 7 (2014), 733-744.

\bibitem{TU142} E.V. Teixeira, J.M. Urbano, \textit{An intrinsic Liouville theorem for degenerate parabolic equations}, Arch. Math. 102 (2014), 483-487.

\bibitem{U08} J.M. Urbano, \textit{The method of intrinsic scaling}, A systematic approach to regularity for degenerate and singular PDEs, Lectures Notes in Mathematics 1930, Springer-Verlag, Berlin, 2008.

\bibitem{W86} M. Wiegner, \textit{On $C^{\alpha}$-regularity of the gradient of solutions of degenerate parabolic systems}, Ann. Mat. Pura Appl. 145 (4) (1986), 385-405.

\end{thebibliography}
\end{document}